\documentclass[11pt,letterpaper]{amsart}
\usepackage{amsfonts, amsthm, amssymb, amsmath, stmaryrd}
\usepackage{thmtools}
\usepackage[bookmarksnumbered,bookmarksopen]{hyperref}
\usepackage{cleveref}
\usepackage{scalerel,xcolor}
\usepackage{mathrsfs,array}
\usepackage{eucal,color,times,enumerate,accents}
\usepackage{tikz-cd}
\usepackage{bbm}
\usepackage[utf8]{inputenc}
\usepackage{pgfplots}
\pgfplotsset{compat=newest}
\usepackage{lipsum}
\usepackage{graphicx}
\usepackage{rotating}
\usepackage{enumitem}
\usepackage{enumerate}
\usepackage{footmisc}
\usepackage{textcomp}

\hypersetup{colorlinks, linkcolor=black, citecolor=black}
\usepackage[colorinlistoftodos]{todonotes}
\usepackage[left=4cm,right=4cm,top=3cm,bottom=3cm]{geometry}
\usepackage{mathtools}

 \numberwithin{equation}{section}
\def\darrow{\mathrel{\ThisStyle{\ooalign{$\SavedStyle\rightarrow$\cr%
  \hfil\textcolor{white}{\rule{2\LMpt}{1\LMex}}\kern2\LMpt\hfil}}}}
\newtheorem{theorem}[equation]{Theorem}

\newtheorem{proposition}[equation]{Proposition}
\newtheorem{lemma}[equation]{Lemma}

\newtheorem{corollary}[equation]{Corollary}

\theoremstyle{definition}
\newtheorem{definition}[equation]{Definition}

\newtheorem{remark}[equation]{Remark}
\newtheorem{example}[equation]{Example}

\usepackage{tikz}
\usepackage{tikz-cd}
\usepackage{rotating}
\usetikzlibrary{arrows,shapes,positioning}
\usetikzlibrary{decorations.markings}
\tikzstyle arrowstyle=[scale=1]
\usepackage{lipsum}
\usepackage{url}
\usepackage{imakeidx}
\makeindex[title=Index of definitions and notations,intoc]

\newcommand{\Dim}{\textrm{dim}}
\newcommand{\Ker}{\mathrm{Ker}}

\newcommand{\Spec}{\mathrm{Spec}}

\newcommand{\Edge}{\mathrm{Edge}}

\newcommand{\crys}{\mathrm{crys}}

\newcommand{\dR}{\mathrm{dR}}
\newcommand{\Proj}{\mathrm{Proj}}
\newcommand{\CH}{\mathrm{CH}}
\newcommand{\HH}{\mathrm{H}}

\newcommand{\cl}{\mathrm{cl}}

\numberwithin{equation}{section}
\DeclareFontFamily{U}{boondoxuprscr}{\skewchar \font =45}
\DeclareFontShape{U}{boondoxuprscr}{m}{n}{
    <-> BOONDOXUprScr-Regular}{}
\DeclareFontShape{U}{boondoxuprscr}{b}{n}{
    <-> BOONDOXUprScr-Bold}{}
\newcommand\emme{\text{\fontencoding{U}\fontfamily{boondoxuprscr}\fontshape{n}\selectfont m}}
\begin{document}
\title{An Artin--Mumford criterion for conic bundles in characteristic two}
\author{Emiliano Ambrosi and Giuseppe Ancona}
\begin{abstract}
	We prove a characteristic two version of the famous criterion of Artin and Mumford for irrationality of conic bundles. On the one hand, combined with the pathological behaviour of conic bundles in characteristic two, this allows us to construct easier and more explicit examples of irrational conic bundles. On the other hand, degeneration techniques \textit{à la Voisin} allow to deduce irrationality results in characteristic zero. 
\end{abstract}
\maketitle
\tableofcontents
\section{Introduction}
\subsection*{The Artin--Mumford criterion for irrationality of conic bundles}
One of the fundamental problem in algebraic geometry is to understand whether a variety $X$ over an algebraically closed field $k$ is stably rational, i.e. whether $X\times \mathbb P^n$ is birational to $\mathbb P^m$ for some $m,n\in \mathbb N$. One of the first and most influential stable irrationality criterion, has been proved by Artin and Mumford. Since it is the starting point of this paper, we recall it now. 
 \begin{theorem}\cite{AM}\label{thm:AM}
	Let $k$ be an algebraically closed field of characteristic different from $2$. Let  $f:X\rightarrow B$ be a flat conic bundle over $k$  between smooth, proper and connected $k$-varieties. 
Assume the following.
\begin{enumerate}
	 \item The discriminant is disconnected.
	 \item There are two distinct connected components of the discriminant on which all the fibers are reduced and such that the conic bundle over each of them is not a product.
	\item The group $H_{\textrm{\'et}}^3(B,\mathbb Z_2)$ vanishes.
	\end{enumerate}
	 Then   $X$ is not stably rational.
\end{theorem}
Since then, many other techniques have been developed to study rationality problems and we focus here on two main approaches.

On the one hand, Theorem \ref{thm:AM} has been extended  via the introduction of unramified cohomology (with $\mathbb Z/2$-coefficients) by Colliot-Th\'{e}l\`ene and Ojanguren \cite{ColliotOj}. This has been further developed and combined with degenerations techniques \textit{à la Voisin} \cite{Voisin,CTPir}, leading to striking applications (culminating with the work of Schreieder \cite{Stefan}). These techniques proved  to be very powerful, although constructing examples with them becomes tricky.   
Also, these results do not work in characteristic 2, since $2$-torsion \'etale cohomology is quite badly behaved.

On the other hand, techniques using reduction modulo a prime number $p$, especially with  $p=2$, have been developed by Kollar \cite{Kollar} and used by Totaro  \cite{Totaro}, by exploiting pathological behaviour of differential forms in positive characteristic. These techniques have the advantage of having a more geometric flavour and of being more elementary in nature. 
\subsection*{The Artin--Mumford criterion  in characteristic 2}
In this paper, we take a first step in trying to combine both approaches, by extending Theorem \ref{thm:AM} to characteristic 2 and showing that this extension is often easier to apply to concrete examples. Moreover, using degeneration techniques  \textit{à la Voisin}, we deduce irrationality results in characteristic zero; see for instance Theorem \ref{teo concreto intro}. 
Our main result is the following.
\begin{theorem}\label{thm:AM2 intro}
	Let $k$ be an algebraically closed field of characteristic $2$. Let  $f:X\rightarrow B$ be a flat and dominant conic bundle over $k$, with smooth generic fiber, between smooth, proper and connected $k$-varieties. 
	Assume the following.
	\begin{enumerate}
		\item The discriminant is disconnected.
	\item There are two disctinct  connected components of the discriminant of $f$, each of which contains a fiber which is not reduced.
	\item The group $\HH^2(B,\Omega^1_{B/k})$ vanishes.
 
	\end{enumerate}
	 Then  $X$ is not stably rational. 
\end{theorem}
\subsection{Comparison between the two Artin--Mumford criteria}
Let us enlighten the main differences between Theorem \ref{thm:AM} and Theorem \ref{thm:AM2 intro}. First of all, in our result, the field $k$ is of characteristic 2, so it covers  the case left open by Theorem \ref{thm:AM}. Secondly, and most importantly, the assumptions (2) are different. In practice, checking assumption (2) in Theorem \ref{thm:AM} is quite complicated since one has to understand the behaviour of the fibers in a family, while assumption (2) in Theorem \ref{thm:AM2 intro} is immediate to check in concrete examples. There is also a version of Theorem \ref{thm:AM2 intro} with reduced fibers over the discriminant (see Theorem \ref{thm:AM2}), which is more similar to the original statement of Artin and Mumford, but one shall consider Theorem \ref{thm:AM2 intro} as more useful in practice, because no hypothesis on the nontriviality of the family appears. 

\

One could also state and prove Theorem \ref{thm:AM2 intro} in other characteristics,   but constructing nonreduced fibers there is much harder. Indeed, locally in $B$, the conic bundle has equation 
 \[\alpha a^2+\beta b^2+\gamma c^2+\alpha' bc+\beta' ac+\gamma' ab=0\]
 for some local functions $\alpha,\beta,\gamma,\alpha',\beta',\gamma'$ on $B$. In characteristic two, the locus of nonreduced fiber is simply given by the equations $\alpha'=\beta'=\gamma'=0$, while in characteristic different from $2$ it has a complicated expression (see \cite[Definition 3.6]{Tanaka}). 
 
 \
 
 The invariant used by Artin and Mumford to obstruct rationality is the $2$-torsion in \'etale cohomology. It turns out that the nonzero class they construct is algebraic. In our case we follow their approach and construct a similar $2$-torsion algebraic class. The only cohomology in characteristic two which might have nonzero $2$-torsion is crystalline cohomology, hence we have to work there. This adds various technical complications, since one has to deal with differential forms on conic bundles in characteristic 2. To do this, we will need to work with complexes and to control the difference between relative and asbolute diffrential forms (see Proposition \ref{prop : tecnica}). 
 
\subsection*{Applications}
We construct explicit examples to which Theorem \ref{thm:AM2 intro} applies. Given the special shape of the discriminant in characteristic 2, this can be achieved quite easily.
\begin{theorem}\label{thm:explicit families}
	Let $k$ be a field of characteristic two. Let $B=\mathbb P^2$ with coordinates $x,y,z$.
Consider the conic bundle $f:X\rightarrow B$ defined in  $\mathcal O \oplus \mathcal O(1) \oplus  \mathcal O(3)$ by

 $$
  a^2 +
  x  ab + yz b^2 +
  (x(y^3+z^3)+y^2z^2) bc + (y^6+z^6+x^4yz+xz^5+xy^5) c^2 =0.
$$
Then   $X$ is not stably rational.  
\end{theorem}

From Theorem \ref{thm:explicit families} and the other examples we construct, we deduce the following.
\begin{theorem}\label{teo concreto intro}
	Let $k$ be a field of characteristic zero or two. A very general conic bundle in $\mathbb P (\mathcal O \oplus \mathcal O(1) \oplus  \mathcal O(3))$ over $\mathbb P_k^2$ with values in $\mathcal O$ has no decomposition of the diagonal, hence it is not stably rational. 
	
	Similarly, a very general conic bundle in $\mathbb P (\mathcal O(1) \oplus \mathcal O(1) \oplus  \mathcal O(3))$ over $\mathbb P_k^2$ with values in $\mathcal O(1)$ has no decomposition of the diagonal, hence  it is not stably rational.
\end{theorem}
As far as we are aware of, this result is new in characteristic two, but it seems likely that it can be proved  using other obstructions in characteristic zero. We believe that much more interesting examples can be constructed in characteristic two, hence giving new results in  characteristic zero, see for instance Remark \ref{rem ex double line}.
\begin{remark}
A different irrationality criterion for conic bundle in characteristic 2 has been proved in \cite{Auelconic}. Compared to the geometric  Theorem \ref{thm:AM2 intro}, their criterion is closer to the unramified cohomological interpretation of Theorem \ref{thm:AM} given in \cite{ColliotOj}. Since unramified cohomology with 2-torsion coefficients does not work well in characteristic 2, their criterion is quite involved to check in practice and it needs the explicit construction of an auxiliary conic bundle. In Example \ref{ex:Auel}, we show how to interpret their main concrete application via Theorem \ref{thm:AM2 intro}.
 \end{remark}
\subsection*{Organization of the paper}  
Section 2 contains some  conventions. Section 3 recalls the definition of conic bundle and the one of its discriminant. This is  subtle in characteristic two. 
Section 4 is a technical section doing computations on the cohomology of   sheaves of differential forms on conic bundles. These computations are used in Section 5 where the crystalline version of the  Artin--Mumford criterion is stated and proved. 
Section 6 explains a birational transformations which makes the disconnectedness hypothesis  (1) of  Theorem \ref{thm:AM2 intro}  easier to be checked. 
 In Section 7 we recall the notion of the decomposition of the diagonal, its relation to rationality problems and explain why torsion in crystalline cohomology obstructs the decomposition of the diagonal. Putting this discussion together with the crystalline Artin--Mumford criterion, we give   geometric conditions that obstruct rationality. Finally, in Section 8 we give explicit examples where these geometric conditions are satisfied and deduce irrationality results for some very general varieties in characteristic zero and two. We also discuss there examples one might hope to construct  in the future.
\subsection*{Acknowledgements}  We thank Olivier Benoist, Nirvana Coppola, Fabio Bernasconi, Francesc Fit\'e, Lie Fu, Mirko Mauri, John Ottem, Stefan Schreieder and Enric Florit Zacarias for useful discussions.

This research was partly supported by the grant ANR--23--CE40--0011 of \emph{Agence National de la Recherche}. 
\section{Notation and convention}
Throughout the paper we use the following conventions.
\begin{enumerate}
\item A variety is a separated scheme of finite type over a field. When the  field $k$ is specified we call this a $k$-variety.
\item A conic over a field $k$ is a subscheme of $\mathbb P_{k}^2$ defined by a global section of $\mathcal O(2)$. 
\item A conic which is geometrically isomorphic to $xy=0$ in $\mathbb P^2$, is a cross.   A conic which is geometrically isomorphic to $x^2=0$ in $\mathbb P^2$, will be called a double line. See also \Cref{def discr}.
\item If $k$ is a field and $\mathcal F$ is a coherent sheaf on a proper $k$-scheme $S$, we write $h^i_S(\mathcal F)$ for the dimension of the $k$-vector space $\HH^i(S,\mathcal F)$ and we drop the index $S$ if it is clear from the context. 
\item If $X$ is a $k$-variety and $\mathcal E$ is vector bundle of rank $r$ on $X$, we write $\mathbb P(\mathcal E)\rightarrow X$ for the associated projective bundle with fibers $\mathbb P^{r-1}$. 
\item The Chow groups $\CH(X)$ of algebraic cycles modulo rational equivalence of a  variety $X$ has always integral coefficients.
\item For an object $X$ (a variety, a sheaf,$\ldots$) living on some base $B$, we write $X_S$ for the base change of $X$ to $S$ when the map from $S$ to $B$ is clear.
\item For a prime number $p$ and an abelian group $M$, the $p$-torsion of $M$ is denoted by $M[p]$.
\item If $X$ is a variety over a field $k$ of positive characteristic, $\HH^i_{\crys}(X):=\HH^i_{\crys}(X,W(k))$ denotes the crystalline cohomology with integral coefficients. 
\item The projective plane will sometimes be the base and sometimes be a fiber of a fibration. With the hope of clearly distinguishing the two situations, we use $a,b,c$ as variables on the fiber and $x,y,z$ as variables on the base.
\end{enumerate}

\section{Discriminant of conic bundles in characteristic two}
This section gives generalities on conic bundles with particular emphasis to characteristic two. This is all taken from \cite{Auelconic} and \cite{Tanaka} and recalled here for the convenience of the reader. We fix  an algebraically closed field $k$  and  a connected smooth $k$-variety  $B$. 
\subsection*{Quadratic forms on rank $3$ vector bundles}
Let $\mathcal E$ be a rank 3-vector bundle on $B$.
 Let $S_2(\mathcal E)\subset \mathcal E\otimes \mathcal E$ be the set of symmetric tensors and $\mathcal E\otimes \mathcal E\twoheadrightarrow S^2(\mathcal E)$ the symmetric power of $\mathcal E$. By construction, there are identifications $S_2(\mathcal E)^{\vee}\simeq S^2(\mathcal E^{\vee})$.
\begin{definition}
	Let $\mathcal L$ be a line bundle on $B$. A quadratic form on $\mathcal E$ with values in $\mathcal L$ is a morphism $q:S_2(\mathcal E)\rightarrow \mathcal L$. 
\end{definition}
By duality, giving a quadratic form on $\mathcal E$ with values in $\mathcal L$ is equivalent to give a global section of $S^2(\mathcal E^{\vee})\otimes \mathcal L$. 

\begin{remark}Classically, quadratic forms are maps $q:\mathcal E\rightarrow \mathcal L$  satisfying $q(av)=a^2v$ for local sections $a$ of $\mathcal O_B$ and $v$ of $\mathcal E$ and such that the map $\varphi:\mathcal E\times \mathcal E\rightarrow \mathcal L$ defined on local sections by $\varphi(x,y):=q(x+y)-q(y)-q(x)$ is bilinear. To recover this notion from the one defined before, one observes that the natural map sending  $q:S_2(\mathcal E)\rightarrow \mathcal L$ to the map $\widetilde q:E\rightarrow \mathcal L$ defined on local section by $\widetilde q(v)=q(v\otimes v)$, gives a bijection between classical quadratic forms and the one we defined (see e.g. \cite[Proposition 2.6.1]{Wood} or \cite[Lemma 1.1]{Auelclifford}).
\end{remark}


\subsection*{Conic bundles}
Let $q:S_2(\mathcal E)\rightarrow \mathcal L$ be a quadratic form on a vector bundle $\mathcal E$  on $B$ with values in a line bundle $\mathcal L$ on $B$ and let $\pi:\mathbb P(\mathcal E)\rightarrow B$ be the canonical map. 
Since 
$$\HH^0(B, S^2(\mathcal E^{\vee})\otimes \mathcal L)\simeq \HH^0(B, \pi_*(\mathcal O_{\mathbb P(\mathcal E)}(2)\otimes \pi^*\mathcal L))\simeq \HH^0(\mathbb P(\mathcal E), \mathcal O_{\mathbb P(\mathcal E)}(2)\otimes \mathcal \pi^*L),$$
every nonzero quadratic form $q:S_2(\mathcal E)\rightarrow \mathcal L$ defines a closed subscheme $X_{q}\subset \mathbb P(\mathcal E)$ whose natural morphism $f_q:X_q\rightarrow B$ has generic fiber a conic. 
\begin{definition}\label{def conic bundle}
Let  $f:X\rightarrow B$ be a morphism between $k$-varieties. 
 	We say that $f$
is a conic bundle if there exists a rank 3 vector bundle $\mathcal E$ on $B$, a line bundle $\mathcal L$ on $B$ and a quadratic form $q:S_2(\mathcal E)\rightarrow \mathcal L$ such that $f:X\rightarrow B$ is isomorphic to $f_q:X_q\rightarrow B$. In this case, we say that $X$ is a conic bundle in $ \mathbb P(\mathcal E)$ over $B$ with values in $\mathcal{L}$.
\end{definition}
\subsection*{Discriminant and locus of double lines}
Let $a,b,c$ be the coordinates in $\mathbb P^2$. Over an algebraically closed field, up to a change of coordinates, a conic $C$ is defined by one of the following equations  (see e.g. \cite[Corollary 2.5]{Auelconic}):
$$\text{ (1) } \quad a^2+bc;\quad \text{ (2) }\quad  ab;\quad\text{ (3) }\quad  a^2.\quad$$

In case (1), $C$ is smooth. In case (2), $C$ it is the wedge of two $\mathbb P^1$. In case (3), $C$ is irreducible but not reduced.

\begin{definition}\label{def discr}
With the above notations, a conic of type (2) will be called a cross and a conic of type (3) will be called a double line.

Let $f:X\rightarrow B$ be a flat conic bundle. We will denote by $\Delta$ the discriminant of $f$ (the locus of singular fibers) and by $\Sigma$ the locus of double lines of $f$.
\end{definition}
\begin{remark}
	In  \cite[Definition 3.4-3.10]{Tanaka}, Tanaka defined closed subschemes  $\Sigma\subseteq \Delta\subseteq B$  such that $X_b$ is singular if and only if $b\in \Delta$ (i.e. the geometric fiber is of type (2) or (3)) and the geometric fiber is moreover of type (3) if and only if $b\in \Sigma$. 
In general, both  $\Delta$ and  $\Sigma$ might be nonreduced or might  be equal to $B$. If $X\rightarrow B$ is generically smooth, then $\Delta$ is a Cartier divisor in $B$.
For our purposes, we will only need  the reduced structure of $\Delta$ and  $\Sigma$.
\end{remark}
 
\subsection*{Direct sum of line bundles}
Assume that $\mathcal E\simeq \mathcal E_a\oplus \mathcal E_b \oplus \mathcal E_c$ is the direct sum of three line bundles $\mathcal E_i$ and write $\mathcal E_{i,j}:=\mathcal E_i\otimes \mathcal E_j$. In this case we have
$$S_2(\mathcal E)\simeq \mathcal E_{a,a}\oplus \mathcal E_{b,b}\oplus \mathcal E_{c,c}\oplus\mathcal E_{a,b} \oplus\mathcal E_{a,c}\oplus\mathcal E_{b,c}$$
and
$$S_2(\mathcal E^{\vee})\simeq \mathcal E^{\vee}_{a,a}\oplus \mathcal E^{\vee}_{b,b}\oplus \mathcal E^{\vee}_{c,c}\oplus\mathcal E^{\vee}_{a,b} \oplus\mathcal E^{\vee}_{a,c}\oplus\mathcal E^{\vee}_{b,c}$$
Hence to give a quadratic form on $\mathcal E$ with values in $\mathcal L$, is equivalent to give a collection of six global sections $s_{i,j}\in \HH^0(B,\mathcal E_{i,j}^{\vee}\otimes \mathcal L)$ for $i\leq j$. We represent this situation with a "half" matrix:
$$
\begin{matrix}
&\vline & \mathcal E^{\vee}_{a} & \mathcal E^{\vee}_{b} & \mathcal E^{\vee}_c\\
\hline
\mathcal E^{\vee}_{a} &\vline & s_{a,a} &  & \\
\mathcal E^{\vee}_{b}&\vline & s_{a,b}& s_{b,b}&\\
\mathcal E^{\vee}_c &\vline & s_{a,c}& s_{b,c}& s_{c,c}
\end{matrix}
$$
corresponding to the conic bundle of equation
$$
  s_{a,a} a^2+
	   s_{b,b} b^2+s_{c,c} c^2 +  s_{a,b} ab +
  s_{a,c} ac +  s_{b,c} bc =0.
 $$
In this case the discriminant divisor is given by the zero locus of 
$$\Delta= 4s_{a,a}s_{b,b}s_{c,c}+ s_{a,b}s_{b,c} s_{a,c}-s_{a,b}^2s_{c,c}-s_{a,c}^2s_{b,b}-s_{b,c}^2s_{a,a}=0.$$
The locus of double lines $\Sigma$ is much more complicated in general. See \cite[Proposition 3.12]{Tanaka} for an explicit formula. 
\begin{remark}\label{discr car 2}
When $k$ is of characteristic $2$, $\Delta$ simplifies to 
$$ s_{a,b}s_{b,c} s_{a,c}+s_{a,b}^2s_{c,c}+s_{a,c}^2s_{b,b}+s_{b,c}^2s_{a,a}=0.$$
Moreover, the locus $\Sigma$ of double lines is simply given by  the equation $$s_{a,b}=s_{a,c}=s_{b,c}=0.$$
These simplified formulas are the reason for which it is easier to construct examples where   \Cref{thm:AM2 intro} applies than where the original Artin--Mumford criterion applies. 
\end{remark}

\section{Cohomology of conic bundles}
In this section we collect preliminaries on the cohomology of conics and conic bundles. 
We will fix   an algebraically closed field $k$  and  a connected smooth $k$-variety  $B$.
  \begin{lemma}\label{lem : conic}

		Let $i:C  \hookrightarrow   \mathbb P_{k}^2$ be a conic defined by a sheaf of ideal $\mathcal I_C\simeq \mathcal O(-2)$. Then the following computations hold.
			\begin{enumerate}
	\item $h^0_C(\mathcal O_C)=1$ and $h^1_C(\mathcal O_C)=0$.
\item $h^0_C(i^*\Omega^1_{\mathbb P^2/k})=0$ and $h^1_C(i^*\Omega^1_{\mathbb P^2/k})=4$.
\item $h^0_C(i^*\mathcal I_C)=0$ and  $h^1_C(i^*\mathcal I_C)=3$.
	\end{enumerate}
\end{lemma}
\proof
 Part (1)  follows from the exact sequence of coherent sheaves on $\mathbb P^2$
	\begin{equation}\label{eq : defcon}
	0\rightarrow \mathcal O_{\mathbb P^2}(-2)\rightarrow \mathcal O_{\mathbb P^2}\rightarrow i_* \mathcal O_{C}\rightarrow 0	\end{equation}
	 and the fact that
	$$h^i_{\mathbb P^2}(\mathcal O_{\mathbb P^2}(-2))=0 \text{ for } i\geq 0, \quad h^0_{\mathbb P^2}(\mathcal O_{\mathbb P^2})=1,\quad h^1_{\mathbb P^2}(\mathcal O_{\mathbb P^2})=h^2_{\mathbb P^2}(\mathcal O_{\mathbb P^2})=0.$$
	
	Now recall the Euler exact sequence
	$$0\rightarrow \Omega^1_{\mathbb P^2/k}\rightarrow \mathcal O_{\mathbb P^2}(-1)^3\rightarrow \mathcal O_{\mathbb P^2}\rightarrow 0.$$
	Pulling back to $C$, we find an exact sequence 
	$$0\rightarrow i^*\Omega^1_{\mathbb P^2/k}\rightarrow i^*\mathcal O_{\mathbb P^2}(-1)^3\rightarrow \mathcal O_{C}\rightarrow 0$$ 
	so that, thanks to (1), part (2) is reduced to show that 
	$$h^0_C(i^*\mathcal O_{\mathbb P^2}(-1))=0, \quad \text{and}\quad  h^1_C(i^*\mathcal O_{\mathbb P^2}(-1))=1.$$
	These in turn, follows from the exact sequence
	$$0\rightarrow \mathcal O_{\mathbb P^2}(-3)\rightarrow \mathcal O_{\mathbb P^2}(-1)\rightarrow i_*\mathcal O_{C}(-1) \rightarrow 0,$$
	obtained by tensoring (\ref{eq : defcon}) with $\mathcal O_{\mathbb P^2}(-1)$, and the fact that 
		$$\quad h^i_{\mathbb P^2}(\mathcal O_{\mathbb P^2}(-1))=0 \text{ for } i\geq 0,\quad h^0_{\mathbb P^2}(\mathcal O_{\mathbb P^2}(-3))=h^1_{\mathbb P^2}(\mathcal O_{\mathbb P^2}(-3))=0, \quad h^2_{\mathbb P^2}(\mathcal O_{\mathbb P^2}(-3))=1.$$

For (3), tensoring (\ref{eq : defcon}) with $\mathcal I_C \simeq  \mathcal O_{\mathbb P^2}(-2)$ we get an exact sequence 
$$0\rightarrow \mathcal O_{\mathbb P^2}(-4)\rightarrow \mathcal O_{\mathbb P^2}(-2)\rightarrow i_*i^*\mathcal I_C\rightarrow 0,$$
so the conclusion follows from the equalities
$$h^i_{\mathbb P^2}(\mathcal O_{\mathbb P^2}(-2))=0 \text{ for } i\geq 0, \quad h^0_{\mathbb P^2}(\mathcal O_{\mathbb P^2}(-4))=h^1_{\mathbb P^2}(\mathcal O_{\mathbb P^2}(-4))=0,\quad h^2_{\mathbb P^2}(\mathcal O_{\mathbb P^2}(-4))=3.$$
\endproof

We now turn to  some   computations on the cohomology conic bundles. We will make an extensive use of the following classical theorem; see \cite[Chapter II, $\S$5, Corollary 2]{Mumford}. 
\begin{theorem}\label{fact : grauert}
Let $f:X\rightarrow B$ be a flat proper morphism and let $\mathcal F$ be a coherent sheaf over $X$ flat over $B$. If $h^i_{X_p}(\mathcal F_p)$ is constant for every geometric point $p\in B$, then $R^if_*\mathcal F$ is locally free and the natural map
$$R^if_*\mathcal F\otimes k(p)\rightarrow \HH^i(X_p,\mathcal F_p)$$
is an isomorphism. 
\end{theorem}
\begin{proposition}\label{prop : tecnica}
Let $X$ be a smooth $k$-variety and let $f:X\rightarrow B$ be a flat conic bundle with smooth generic fiber. Then 
	\begin{enumerate}
		\item $f_*\mathcal O_X=\mathcal O_B$ and $R^if_*\mathcal O_X=0$ for $i>0$;
		\item The right exact sequence 
\[			f^*\Omega^1_{B/k}\rightarrow \Omega^1_{X/k}\rightarrow \Omega^1_{X/B}\rightarrow 0.
\]	is also left exact. 
	\item $f_*\Omega^1_{X/B}=0$
		\item The natural maps 
		$$\Omega^1_{B/k}\rightarrow f_*\Omega^1_{X/k}\quad \text{and}\quad R^1f_*\Omega^1_{X/k} \rightarrow   R^1f_*\Omega^1_{X/B}$$  
		are isomorphism.
	\end{enumerate}
\end{proposition}
\proof
 
 Part (1)  follows directly from Theorem \ref{fact : grauert} and Lemma \ref{lem : conic}(1) (see also \cite[Lemma 2.5]{Tanaka}).
 
 For (2), we have to show that the map $f^*\Omega^1_{B/k}\rightarrow \Omega^1_{X/k}$ is injective. Since both sheaves are locally free, it is enough to check injectivity after the restriction to  an open subset of $X$. As $f:X\rightarrow B$ is generically smooth, we can assume that $f$ is smooth and in this case the sequence is also left exact (see e.g. \cite[Tag 02K4]{stacks-project}).
 
Let us now show (3). Let $\mathcal E$ be a rank 3 vector bundle on $B$ as in \Cref{def conic bundle}.  In particular, there is a closed immersion $i:X\rightarrow \mathbb P(\mathcal E)$ over $B$, such that, for  every geometric point $p\in B$, the inclusion $i_p:X_p\rightarrow \mathbb P(\mathcal E)_p=\mathbb P^2$ is the anticanonical embedding of $X_p$. Let $\mathcal I_X$ be the sheaf ideal of $X$ in $\mathbb P(\mathcal E)$. We have a right exact sequence
	$$i^*\mathcal I_X\rightarrow i^*\Omega^1_{\mathbb P(\mathcal E)/B}\rightarrow \Omega^1_{X/B}\rightarrow 0$$
	by \cite[Tag 01UZ]{stacks-project}.
	We claim that it is also left exact, i.e. that $i^*\mathcal I_X\rightarrow i^*\Omega^1_{\mathbb P(\mathcal E)/B}$
	is injective. Indeed, since $X\subseteq \mathbb P(\mathcal E)$ is a local complete intersection, $i^*\mathcal I_X$ is locally free (\cite[Tag 06B9]{stacks-project}). Since also  $\Omega^1_{\mathbb P(\mathcal E)/B}$ is locally free, it is enough to show that $i^*\mathcal I_X\rightarrow i^*\Omega^1_{\mathbb P(\mathcal E)/B}$ is injective on an open subset. Hence,   we can assume that $f$ is smooth, where the conclusion follows from \cite[Tag 06AA]{stacks-project}. 
	
We can now pushforward via $f:X\rightarrow B$ the short exact sequence
		$$0\rightarrow i^*\mathcal I_X\rightarrow i^*\Omega^1_{\mathbb P(\mathcal E)/B}\rightarrow \Omega^1_{X/B}\rightarrow 0$$
		to find a long exact sequence
		\[
			0\rightarrow f_*i^*\mathcal I_X\rightarrow f_*i^*\Omega^1_{\mathbb P(\mathcal E)/B} \rightarrow f_*\Omega^1_{X/B}\rightarrow R^1 f_*i^*\mathcal I_X\rightarrow R^1f_*i^*\Omega^1_{\mathbb P(\mathcal E)/B}\rightarrow R^1f_*\Omega^1_{X/B}\rightarrow 0.	\]
							By  Theorem \ref{fact : grauert}  and Lemma \ref{lem : conic}(2), $f_*i^*\Omega^1_{\mathbb P(\mathcal E)/B}$   vanishes. 
				Hence $f_*\Omega^1_{X/B}$ identifies with the kernel of the map				$$R^1f_*i^*\mathcal I_X\rightarrow R^1f_*i^*\Omega^1_{\mathbb P(\mathcal E)/B}.$$
				We want to show that the map is injective. Since $X\rightarrow B$ is flat and $\mathcal I_X$ is flat over $B$, the restriction of $\mathcal I_X$ to $\mathbb P(\mathcal E)_p$  identifies with the ideal defining $X_p$ in $\mathbb P(\mathcal E)_p$. Hence, by Lemma \ref{lem : conic} and Theorem \ref{fact : grauert}, the coherent sheaves $R^1f_*i^*\mathcal I_X$ and $R^1f_*i^*\Omega^1_{\mathbb P(\mathcal E)/B}$
				are locally free, so it is enough to show injectivity on an open, hence to show the vanishing of $f_*\Omega^1_{X/B}$ on an open. We can then assume that $f$ is smooth, in which case the fibers of $f_*\Omega^1_{X/B}$ identifies with  $\HH^0(\mathbb P^1, \Omega^1_{\mathbb P^1/k})=0$, so we conclude 
				 again by Theorem \ref{fact : grauert}.

	Finally, let us now show (4). By (2), we have a short exact sequence 	$$0\rightarrow f^*\Omega^1_{B/k}\rightarrow \Omega^1_{X/k}\rightarrow \Omega^1_{X/B}\rightarrow 0.$$
	Pushing forward, we get a long exact sequence
	$$0\rightarrow f_*f^*\Omega^1_{B/k}\rightarrow f_*\Omega^1_{X/k}\rightarrow f_*\Omega^1_{X/B}\rightarrow R^1f_*f^*\Omega^1_{B/k}\rightarrow R^1f_*\Omega^1_{X/k}\rightarrow R^1f_*\Omega^1_{X/B}\rightarrow 0.$$
	By the projection formula and (1)
	$$f_*f^*\Omega^1_{B/k}\simeq \Omega^1_{B/k}\otimes f_*\mathcal O_X\simeq \Omega^1_{B/k} \quad \text{and}\quad R^1f_*f^*\Omega^1_{B/k}\simeq \Omega^1_{B/k}\otimes R^1f_*\mathcal O_X=0,$$
	hence $R^1f_*\Omega^1_{X/k}\simeq R^1f_*\Omega^1_{X/B}$ and there is a short exact sequence
	$$0\rightarrow \Omega^1_{B/k}\rightarrow f_*\Omega^1_{X/k}\rightarrow f_*\Omega^1_{X/B}\rightarrow 0.$$
	So the conclusion follows from (3). 
\endproof
\begin{corollary}\label{rottura di cazzo}
Keep notation from the above theorem.  Then the following holds.
	\begin{enumerate}
		\item $ h^i_B(\mathcal O_B) = h^i_X(\mathcal O_X)$  for all $i>0$ and  $h^0_B(\Omega^1_{B/k})=  h^0_X(\Omega^1_{X/k})$.
		\item The natural map $  \HH^0(B,\Omega^i_{B/k}) \rightarrow \HH^0(X,\Omega^i_{X/k})$ is injective for all $i>0$.
		\end{enumerate}
\end{corollary}
\proof
  Part (1) follows from the Leray spectral sequence for $f:X\rightarrow B$ and Proposition \ref{prop : tecnica}(1-3-4).
  
  Let us now show (2). By Proposition \ref{prop : tecnica}(2) the natural map  
  \[f^*\Omega^1_{B/k} \hookrightarrow\Omega^1_{X/k}.\]
  is injective. Since  these sheaves   are locally free, we can pass to exterior powers and deduce that
 the natural map  
 \[f^*\Omega^i_{B/k} \hookrightarrow\Omega^i_{X/k}.\]
  is injective. By the projection formula and Proposition \ref{prop : tecnica}(1), taking global sections gives the conclusion.
\endproof

\begin{corollary}\label{cor : sembrapocointeressante}
Keep notation from the above theorem. Assume in addition that $h^i_B(\mathcal O_B)$ and $h^0_B(\Omega^i_{B/k}) $ vanish for $i>0$. Then the following holds.
\begin{enumerate}
	\item $h^i_X(\mathcal O_X)=0$ for every $i>0$ and $h^0_X(\Omega^1_{X/k})=0$
	\item If one has also $h^0_X(\Omega^i_{X/k})=0$ for $i>0$ then $\HH^1(X,\Omega^1_{X/k})=\HH^2_{\dR}(X)$.
\end{enumerate}
\end{corollary}
\proof
Part (1) follows from \Cref{rottura di cazzo}(1).

Let us now show (2).  By (1) and the assumption $h^0_X(\Omega^i_{X/k})=0$ it is enough to show that $h^1_{X}(\Omega^1_{X/k})=\Dim_k(\HH^2_{\dR}(X))$.
Let $\mathcal H_X^i$ be the sheaf $\HH^i(\Omega_X^{\bullet})$ on $X$, so that there is the conjugate spectral sequence
$$E_2^{a,b}:=\HH^a(X,\mathcal H_X^b)\Rightarrow H_{\dR}^{a+b}(X).$$
The Cartier isomorphism \cite[Theorem 7.2]{Katz} shows that
$$\Dim_k(\HH^a(X,\mathcal H_X^b))=\Dim_k(\HH^a(X,\Omega^b_{X/k}))$$
In particular, by assumption, $E_2^{0,b}$ for $b>0$, so that there are no non trivial morphism from or to $\HH^1(X,\mathcal H_X^1)$, hence $E_{\infty}^{1,1}= \HH^1(X,\mathcal H_X^1)$. Moreover, by hypothesis, 	$E_2^{2,0}$ vanishes as well, so we have $$\HH^1(X,\mathcal H_X^1)\simeq H_{\dR}^{2}(X).$$
To conclude the proof just observe that 
$$\Dim_k(\HH^1(X,\mathcal H_X^1))=h^1_{X}(\Omega^1_{X/k})$$
again by Cartier. 
\endproof
	\section{Crystalline Artin--Mumford criterion  in characteristic two}
	
	In this section we prove a characteristic $2$ version of the Artin--Mumford theorem, stating that, under some hypothesis on the discriminant,   the total space of a conic bundle  has torsion in its cohomology. In Section 7 we will see that such a cohomological consequence is an obstruction to the decomposition of the diagonal, hence to stable rationality.
	
We will make use of the locus of crosses and double lines on the base of a conic bundles, as introduced in \Cref{def discr}.
	
	\begin{definition}\label{def AM}Let $f:X\rightarrow B$ be a flat conic bundle  with smooth generic fiber (\Cref{def conic bundle}). Let $D$ be a closed subvariety    of the discriminant (\Cref{def discr}). We say that $D$ is  Artin--Mumford   if one of the following conditions holds.
		\begin{enumerate}
			\item There exists a point  $p\in D$, such that the fiber  $X_{p}$ is a  double line.
			\item All fibers above  $D$ are crosses, $D$ is smooth and the fibration $X_{D}\rightarrow D$ is not a product.
		\end{enumerate}
		\end{definition}

\begin{theorem}\label{thm:AM2}
	Let $k$ be an algebraically closed field of characteristic $2$. Let  $f:X\rightarrow B$ be a flat conic bundle over $k$ with smooth generic fiber between smooth, proper and  connected $k$-varieties. 
	Assume the follwing.
	\begin{enumerate}
		\item The discriminant is disconnected.
		\item There are two disctinct  connected components of the discriminant of $f$ which are Artin--Mumford  (\Cref{def AM}).
\item The group $\HH^2(B,\Omega^1_{B/k})$   vanishes. 
\item The groups $\HH^i(B,\mathcal O_B)$ and $\HH^0(B,\Omega^i_{B/k}) $ vanish for $i>0$.
\item The groups $\HH^0(X,\Omega^i_{X/k}) $ vanish for $i>0$.
\end{enumerate}
Then $\HH^{2d-2}_{\crys}(X)[2]\neq 0$, where $d$ is the dimension of $X$. 
\end{theorem}
\proof
Let $\alpha\in  \HH^{2d-2}_{\crys}(X)$ be the class from \Cref{def alpha}. It is $2$-torsion by \Cref{sec : alphatorsion}. On the other hand it is nonzero by \Cref{sec : alphanonzero}.
 \endproof
 \begin{remark}
 	The first three hypothesis in the above theorem are  analogous to those of the classical Artin--Mumford criterion, see \Cref{thm:AM}. Hypothesis (4) can probably be avoided with a more complicated proof but it is in practice verified by all the interesting examples (e.g. $B$ rational). Hypothesis (5) is the annoying one, as it concerns the total space and not the base. On the other hand, for applications to rationality problems, both (4) and (5) will disappear, see \Cref{thm:maincriterioninpractice}.
 	\end{remark}
\begin{definition}\label{def half}
	Let $D$ be an Artin--Mumford component in the sense of \Cref{def AM}. We define an half-fiber above $D$ as a $\mathbb{P}^1$ inside $X$ defined in the following way. If $D$ satisfies the hypothesis (1) of \Cref{def AM} we take the fiber with its reduced structure $X^{red}_{p}$. If $D$ satisfies the hypothesis (2) we take any of the two irreducible components of the cross above any closed point in $D$.
	\end{definition}

\begin{remark}
The name half fiber comes from the fact that, their class in cohomology is indeed half of the class of the fiber. This definition a priori does depend on the choice of the point $p$ and not only on $D$, so it is a little abuse to call this \textit{the} half fiber.
\end{remark}

\begin{definition}\label{def alpha}
Keep notation from  \Cref{thm:AM2}. Let $D_1$ and $D_2$ be the two connected components of the discriminant from hypothesis (2) and let $\ell_1$ and $\ell_2$ the associated half fibers as constructed in \Cref{def half}.
Let $\cl: \CH^{d-1}(X)\rightarrow \HH^{2d-2}_{\crys}(X)$ be the cycle class map to crystalline cohomology. Define the class $\alpha \in \HH^{2d-2}_{\crys}(X)$ as $$\alpha:=\cl(\ell_1)-\cl(\ell_2).$$
	\end{definition}

\begin{lemma}\label{sec : alphatorsion}
	The class $\alpha$  is 2-torsion.
	\end{lemma} 
	\proof
We claim  that $2\cl(\ell_i) $ is the class of a fiber. This will conclude the prove as $2 \alpha$ will then be the difference of two fibers, hence zero.

The claim is clear for a double line. In the case of  a cross of lines, let us prove that the assumption on the nontriviality of $X_D\rightarrow D$ implies that $\cl(\ell_i)=\cl(\emme_i)$ where $\emme_i$ is the other $\mathbb P^1$ in the same fiber. (This will imply that $2\cl(\ell_i)=\cl(\ell_i)+\cl(\emme_i)$ is indeed the class of a fiber.)

Let $\pi : \tilde{D_i} \rightarrow D_i$ the double cover trivializing the conic bundle on $D_i$. Let $\tilde{X}$ the normalisation of the pull-back of the conic bundle on $ \tilde{D}_i$. Let us fix one  $\mathbb P^1$ in  $\tilde{X}$ which is sent isomorphically to $\ell_i$ and let us call it $\tilde{\ell}_i.$ In the same fiber as $\tilde{\ell}_i$ the other irreducible component is denoted by $\tilde{\emme}_i$ and it is sent isomorphically to $\emme_i$. 
Let $g$ be the involution on $\tilde{X}$  above $X$.  Now, because the fibration over $D_i$ is not trivial, we have that $g_* \cl(\tilde{\ell}_i)=\cl(\tilde{\emme}_i)$. If we push forward this relation we get
\[\cl(\ell_i) = \pi_* \cl (\tilde{\ell}_i) = (\pi \circ g)_* \cl (\tilde{\ell}_i)  =  \pi_* \cl (\tilde{\emme}_i) = \cl(\emme_i).\]
\endproof
\begin{lemma}\label{sec : alphanonzero}
	The class $\alpha$ from \Cref{def alpha} is nonzero.
	\end{lemma}
\proof
It is enough to show that the image  of the class $\alpha$ via the natural application $\HH^{2d-2}_{\crys}(X)\rightarrow \HH^{2d-2}_{\dR}(X)$ is nonzero. 
To prove that $\alpha$ is nonzero in de Rham cohomology, it is enough to construct a class $\beta\in \HH^{2}_{\dR}(X)$ such that $(\alpha,\beta)=1$, where $(-,-):\HH^{2d-2}_{\dR}(X)\times \HH^{2}_{\dR}(X)$ is the Poincar\'e duality pairing. 
On the other hand,  by \Cref{cor : sembrapocointeressante}, there is an identification $\HH^2_{\dR}(X)=\HH^1(X,\Omega^1_{X/k})$, hence it is enough to construct a class $\beta\in \HH^1(X,\Omega^1_{X/k})$. Such a class $\beta$ can be taken to be as one of the classes appearing in \Cref{lemma beta}(3) and we indeed have $(\alpha,\beta)=1$ by \Cref{lemma cup prod}.
 \endproof

\begin{lemma}\label{lemma beta} 
	
	Let $s:B\rightarrow X$ be a degree 2 multisection of the conic bundle $f$. Define the class  $\beta'\in \HH^1(X,\Omega^1_X)$ as $\beta':=\cl(s(B))$. Consider the edge map 
	$$\Edge:\HH^1(X,\Omega^1_{X/k})\rightarrow \HH^0(B,R^1f_*\Omega^1_X)),$$
	for the Leray spectral sequence for $f:X\rightarrow B$
	$$E_2^{a,b}:=\HH^a(B,R^1f_*\Omega^1_{X/k})\rightarrow \HH^{a+b}(X,\Omega^1_{X/k}).$$
	Then the following holds.
	\begin{enumerate}
	   \item 	Let   $U\subset B$ the open subset on which $f:X\rightarrow B$ is smooth.  Then, the restriction of $\Edge(\beta')$ to $\HH^0(U,R^1f_*\Omega^1_{X_U/k})$ is zero.
	   	\item  There exists a unique class $\widetilde{\beta}\in \HH^0(X,R^1f_{\Omega_{X/k}})$ such that its restriction to  $\HH^0(X-D_1,R^1f_{\Omega_{X/k}})$ vanishes  and its restriction to $\HH^0(X-(\coprod_{i\neq 1}D_i),R^1f_{\Omega_{X/k}})$   equals to $\Edge(\beta')$.
	       	\item $\Edge:\HH^1(X,\Omega^1_{X/k})\rightarrow \HH^0(B,R^1f_*\Omega^1_{X/k}))$ is surjective.  In particular, there exists a class $\beta$ in $\HH^1(X,\Omega^1_{X/k})$ such that $\Edge(\beta)=\widetilde \beta$.
	      	\end{enumerate}
	\end{lemma}
\proof
By Proposition \ref{prop : tecnica}  there is a canonical isomorphism $R^1f_*\Omega^1_{X_U/k}\simeq R^1f_*\Omega^1_{X_U/U}$. In particular, thanks to  Theorem \ref{fact : grauert}, the coherent sheaf $R^1f_*\Omega^1_{X_U/U}$  is locally free so that it is enough to show that the restriction of $\Edge(\beta')$ to the fiber $X_{\eta}$ over the generic point $\eta\in B$ is zero. But this identifies with the class of a point of the smooth conic $X_{\eta}$ defined over a degree 2 extension. Hence it is divisible by $2$, hence it is zero in $\HH^1(X_{\eta},\Omega^1_{X/k(\eta)})$, which proves (1). 	Point (2) follows from  (1) and the assumption on the discriminant.

	  The low degree terms from the Leray spectral  sequence show that the obstruction to the surjectivity of part (3)  is $\HH^2(B,f_*\Omega^1_{X/k})$, which by Proposition \ref{prop : tecnica} is isomorphic to  $\HH^2(B,\Omega^1_{B/k})$, which vanishes by assumption. 
	\endproof

\begin{lemma}\label{lemma cup prod}
 The Poincaré pairing
	$(\alpha,\beta)$ equals $1$. 
\end{lemma}
\proof
Since $(\alpha,\beta)=(\cl(\ell_1),\beta)-(\cl(\ell_2),\beta)$ is it enough to show that $(\cl(\ell_1),\beta)=1$ and $(\cl(\ell_2),\beta)=0$. To do this recall that  $(\cl(\ell_i),\beta)=\beta_{\vert \ell_i}\in \HH^1(\ell_i,\Omega^1_{\ell_i/k})=k$
and that $\beta_{\vert \ell_i}=\Edge(\beta)_{\vert \ell_i}$.  
Now, we have
$$\beta_{\vert \ell_1}=\Edge(\beta)_{\vert \ell_1}=\Edge(\beta')_{\vert \ell_1}=(\cl(s(B)),\ell_1)=1$$ 
since the multisection intersects in 1 point with multiplicity $1$ the irreducible components of $f^{-1}(p_1)^{red}$ (as it intersect with multiplicity $2$ any fiber). 
On the other hand 
$$\beta_{\vert \ell_2}=\Edge(\beta)_{\vert \ell_2}=0$$
since, by construction, $\Edge(\beta)$ vanishes on $\HH^0(B-D_1,R^1f_*\Omega_{X/k})$.
\endproof

\section{Separation of the discriminant divisor}
 In order to apply \Cref{thm:AM2} one needs, among other things, a conic bundle with disconnected discriminant. The goal of this section is the construction of a birational transformation allowing to modify a conic bundle with several irreducible components into one with several connected components. We fix  an algebraically closed field $k$ and  a smooth connected $k$-variety $B$. 
 
 \subsection*{Elementary transformations} We recall here generalities on elementary transformations, see \cite[Section 1]{Maruyama} for details.
 Let $T\subset B$ be a Cartier divisor, $\mathcal F$ be a nonzero vector bundle on $T$ and $g:\mathcal{E}_{\vert T}\rightarrow \mathcal F$ be a surjection. The kernel of $g$ is a  subvector bundle $ \mathcal Y\subset   \mathcal{E}_{\vert T}$. 
The elementary transformation of $\mathcal E$ along $\mathcal Y$ is then the rank 3 subvector bundle $El_{\mathcal Y}(\mathcal E)\subset \mathcal E$ on $B$ given by 
$$El_{\mathcal Y}(\mathcal E):=\Ker(\mathcal E\rightarrow \mathcal E_{\vert T}\rightarrow \mathcal F).$$
The inclusion $El_{\mathcal Y}(\mathcal E)\subset \mathcal E$ induces a birational map  $El_{\mathcal Y}(\mathcal E) \darrow \mathcal E$, which can be explicitly described Zariski locally and it is an isomorphism outside $T$, see \cite[Lemma 1.5]{Maruyama}. 

More precisely, assume that $B=\Spec$ is affine, $\mathbb P(\mathcal E)=\Proj(A[a,b,c])$ and that $T$ is defined by $t=0$ for some $t\in A$. If $\mathcal Y$ is defined by $t,a$, then $\mathbb P(El_\mathcal Y(\mathcal E))$ is defined by $\Proj(A[ta,b,c])$ and the birational map $ \mathbb P(El_\mathcal Y(\mathcal E))   \darrow   \mathbb P(\mathcal E)  $ is induced by sending $a$ to $ta$. If $\mathcal Y$ is defined by $t,a,b$, then $\mathbb P(El_Y(\mathcal E))$ is defined by $\Proj(A[ta,tb,c])$ and the birational map $ \mathbb P(El_Y(\mathcal E)) \darrow  \mathbb P(\mathcal E) $ is induced by sending $a$ to $ta$ and $b$ to $tb$.

 \begin{proposition}\label{separazione}
 	Let $f:X\rightarrow B$ be a generically smooth conic bundle. Let $\Delta$ be the discriminant divisor (\Cref{def discr}). Assume the following.
 	\begin{enumerate} 
 		\item One can write $\Delta = D_1 \cup D_2$ as the union of two closed subvarieties. 
 		\item $D_1$ and $D_2$ are smooth around $ D_1 \cap D_2$ and intersect transversally.
 		\item All the fibers above $D_1\cap D_2$ are crosses.
 	\end{enumerate}
 	Let $P$ be the blow-up of $B$ along $D_1\cap D_2$ and $E$ be the exceptional divisor. 
 Then there exists a conic bundle $g:Y\rightarrow P$ whose discriminant divisor is the (disjoint) union of the strict transforms of  $D_1$ and $D_2$ and such that the restrictions to $P-E$ of $g$ and of the pull-back $f_P:X_P\rightarrow P$ of $f$ coincide.
\end{proposition}
\proof
Let $q:S_2(\mathcal E)\rightarrow \mathcal L$ be the quadratic form corresponding to $f_P:X_P\rightarrow P$. By hypothesis the fibers of the restriction $f_E:X_E\rightarrow E$ are crosses. 
By \cite[Proposition 2.14(2)]{Tanaka}, the set $Y\subset \mathcal P(\mathcal E_{\vert E})$ of points living above $E$ which are singular in the fiber is the projectivization of a rank 1 subvector bundle of $\mathcal F^\vee \subset \mathcal E^\vee_{\vert E}$. Consider  the surjection $\mathcal E_{\vert E}\rightarrow \mathcal F$ and let $\mathcal Y$ be the kernel.

Let $\tilde {\mathcal E}:=El_\mathcal Y(\mathcal E)\subset \mathcal E$ be the elementary transformation of $\mathcal E$ along  $\mathcal Y$  and define $\tilde{q}:S_2(\tilde{\mathcal E})\rightarrow \mathcal L$ to be the restrition of $q$ to $S_2(\tilde{\mathcal E})$.
Let  $s\in \HH^0(P, S^2(\tilde{\mathcal E}^{\vee})\otimes \mathcal L)$ be  the section  corresponding to $\tilde{q}$. We claim that its restriction $s_{\vert E}$ to $E$  vanishes with order exactly $2$. 
Assuming the claim $s$ induce a section in  $\HH^0(B, S^2(\tilde{\mathcal E}^{\vee})\otimes \mathcal L\otimes \mathcal O_P(-2E))$ which in return induces a  quadratic form  $\tilde{q}(E):S_2(\tilde{\mathcal E} \otimes \mathcal O_P(E))\rightarrow \mathcal L$. 
By construction, the new quadratic form $ \tilde{q}(E)$ has discriminant divisor equal to the strict transform of $\Delta$. Moreover, no modification has  been done outside $P-E$, hence this will conclude the proof. 

The claim can be proved locally in a neighborhood of $E$, so we can replace $B$ with the completion of $\mathcal O_{B,D_1\cap D_2}$ and then this with its strictly henselianisation $A$. By \cite[Proposition 2.14(2)]{Tanaka}, we can then assume that $X=\Proj(A[a,b,c]/\alpha c^2-ba)$ for some $\alpha\in A$.

Let $t_1,t_2$ be the local parameter of $D_1$ and $D_2$. By assumption $\alpha=\alpha't_1t_2$, for some $\alpha'\in A^*$. Since $P\subset B\times \mathbb P^1_{\widetilde t_1,\widetilde t_2}$ is defined by $\widetilde t_1t_2=\widetilde t_2t_1$, it will be enough to show that $s$ vanishes with order exactly $2$ on $t_1=t_2=0.$
  By symmetry, it is enough to do the computation in the affine chart $\widetilde t_1=1$. 

In this chart, $X_P$ has equations $\Proj(A[a,b,c]/\alpha'\widetilde t_2t_1^2c^2-ba)$ and $\mathcal Y$ has equations $a=c=t_1=0$. Hence the elementary transformation $\tilde {\mathcal E}$  is \[\Proj(A[a,b,c]/\alpha'\widetilde t_2t_1^2c^2-t_1bt_1a),\] which vanishes with order exactly 2 along $t_1=0$.\endproof
 
 \begin{remark}\label{separation double line}
We do not know a statement analogous to Proposition \ref{separazione} under the assumption that the fibers on the intersection $D_1\cap D_2$ might be nonreduced. A very general statement as Proposition \ref{separazione} cannot be true, but it would be very useful   to find conditions where such a birational separation exists. The main motivation is
the construction of more examples to which \Cref{thm:AM2} applies. For instance,  it is very easy to construct conic bundles with reducible discriminant  and such that all the singular fibers  are nonreduced, see \Cref{rem ex double line}.
\end{remark}
	\section{An obstruction to the decomposition of the diagonal}
In this section we recall the definition of the decomposition of the diagonal, which is intimately related to rationality question, following Voisin \cite{Voisin} and \cite{CTPir}. We give a cohomological obstruction to the decomposition of the diagonal (\Cref{Mimmo} and \Cref{cor Mimmo}) and give geometric settings where this obstruction can be applied (\Cref{thm:maincriterioninpractice}, its corollary and \Cref{cor surf}).
\begin{definition}\label{def dec diag}
Let $X$ be a smooth proper geometrically connected scheme   over a field.    Consider the Chow ring $\CH(X\times X)$ of $X\times X$ with integral coefficients and the class of the diagonal  $\Delta_X \in \CH(X\times X)$ in it.
We say that $X$ has decomposition of the diagonal if there exists a relation 
\[\Delta_X=B_1+B_2\]
in $\CH(X\times X)$ where the projection to the first factor of $B_1$ is supported on a scheme of dimension zero and the projection to the second factor of $B_2$ is not the whole $X$. If such a relation does not exist we say that $X$ has no decomposition of the diagonal.
\end{definition}
The relative version of the previous definition turns out to be the following.
\begin{definition}\label{def CH0}
A proper map $f:X \rightarrow Y$ over a field $k$ is called universally {$\CH_0$-trivial} if, for all field extension $L/k$, the pushfoward map induced by $f$ on Chow groups $(f_L)_*: \CH_0(X_L)\rightarrow \CH_0(Y_L)$ is an isomorphism.

A proper variety $Z$ over $k$ is said to be universally $\CH_0$-trivial  if the structural map $f:Z \rightarrow \Spec(k)$ is so.
\end{definition}
\begin{theorem}\label{diag vois}
	\cite{Voisin,CTPir}.
	\begin{enumerate}
\item The group $\CH_0$ is a birational invariant for smooth proper varieties \cite[Example 16.1.11]{Ful}, in particular being universally $\CH_0$-trivial is a birational invariant for smooth proper varieties over a field.
\item For a smooth proper and geometrically connected variety over a field being universally $\CH_0$-trivial  is equivalent to having the decomposition of the diagonal \cite[Proposition 1.4]{CTPir}. In particular having the decomposition of the diagonal is a birational invariant for smooth proper and geometrically connected varieties. 
\item A variety which has no decomposition of the diagonal is not stably rational \cite[Lemma 1.5]{CTPir}.
\end{enumerate}
\end{theorem}
\begin{definition}\label{degeneration}
Let $X$ and $Y$ be proper  varieties defined over a (possibly different) field. We say that $X$ degenerates to $Y$ if there exist a discrete valuation ring $R$ and  proper faithfully flat scheme  over $  \Spec(R)$ whose generic fiber is $X$ and whose special fiber is $Y$.
\end{definition}
\begin{theorem}\label{CTP}\cite[Theorem 1.14]{CTPir}.
Let $X$ and $Y$ be proper geometrically integral  varieties with $X$ smooth. Suppose tha $X$ degenetes to $Y$ and that  $Y$ admits a desingularization $\pi : \tilde{Y}\rightarrow Y $ such that  $\pi$ is universally $\CH_0$-trivial. Suppose that  $Y$ has no decomposition of the diagonal (in the sense of \Cref{def dec diag}). Then $X$ has no decomposition of the diagonal.
\end{theorem}
 
Motivated by Theorem \ref{thm:AM2}, we now explain how   torsion in crystalline cohomology  obstructs the decomposition of diagonal. This is the combination of   \cite{AmbrosiValloni,AuelBrauer,Rullingbirationalinvariance,Totaro}. 
 \begin{proposition}\label{Mimmo}
			Let $k$ be an algebraically closed field of characteristic $p>0$ and $X$ be a smooth proper and connected $k$-variety. 
		Assume that $\HH^3_{\crys}(X)[p]\neq 0$. Then $X$ has no decomposition of the diagonal.
	\end{proposition}
	\proof
	It is enough to combine  \cite[Theorem 1.1.5]{AmbrosiValloni} with the following Theorem \ref{cime di rape}. 
		\endproof
	\begin{theorem}\label{cime di rape}\cite{AuelBrauer,Rullingbirationalinvariance,Totaro}. 
	Assume that $X$ has decomposition of the diagonal. Then the following holds.
	\begin{enumerate}
		\item $\HH^0(X,\Omega^i_X)=0$ for $i\geq 1$.
		\item $Br(X)=0$.
		\item $\HH^i(X,\mathcal O_X)=0$ for $i\geq 1$.
	\end{enumerate}
	\end{theorem}
		\proof
	Observe that (1) is \cite[Lemma 2.2]{Totaro} and (2) is \cite[Theorem 1.1]{AuelBrauer}. For (3), let $\alpha$ be any class in $\HH^i(X,\mathcal O_X).$ Consider the action of correspondence on Hodge cohomology and let us use the relation from  \Cref{def dec diag}.
	We get 
	\[\alpha=\Delta^*_X \alpha=B^*_1 \alpha+B^*_2 \alpha.\]
  Since the action of $B_1$ factors through a scheme of dimension zero, we get $B^*_1 \alpha=0$. Since   $B_2$ does not dominate the second factor we have $B^*_2 \alpha=0$, by \cite[Proposition 3.2.2]{Rullingbirationalinvariance}. This implies $\alpha=0$ and proves (3).
	\endproof
	
	\begin{lemma}\label{lem : torduality}
		Let $X$ be a smooth proper and connected variety of dimension $d$ over  an algebraically closed field $k$ of characteristic $p>0$. Then the $k$-vector spaces
	 $\HH^{2d-2}_{\crys}(X)[p]$ and $\HH^{3}_{\crys}(X)[p]$ have the same dimension.
	\end{lemma}
\begin{proof}
This follows from Poincar\'e duality in the form stated in \cite[Th\'eor\'eme 2.1.3, Page 555]{Berthelot} and the universal coefficients theorem.
	 	\end{proof}
	\begin{proposition}\label{cor Mimmo}
		Let $X$ be a smooth proper and connected variety of dimension $d$ over  an algebraically closed field $k$ of characteristic $p>0$. 	Assume that $\HH^{2d-2}_{\crys}(X)[p]\neq 0$. Then $X$ has no decomposition of the diagonal.
	\end{proposition}
	\proof
	 This is the combination of \Cref{Mimmo} and  \Cref{lem : torduality}.
		\endproof
	Now \Cref{thm:AM2} can be stated in a more geometric way.
	\begin{theorem}\label{thm:maincriterioninpractice}
	Let $k$ be an algebraically closed field of characteristic $2$ and consider a flat conic bundle $f:V\rightarrow B$  over $k$ with smooth generic fiber between smooth proper $k$-varieties. 
Assume the follwing.
\begin{enumerate}
	\item The discriminant is disconnected.
	\item There are two disctinct  connected components of the discriminant of $f$ which are Artin--Mumford components (\Cref{def AM}).
	\item The group $\HH^2(B,\Omega^1_{B/k})$   vanishes. 
	\end{enumerate}
	Then $V$ has no  decomposition of the diagonal.
	\end{theorem}
	\begin{proof}
		This is the combination of \Cref{thm:AM2} with \Cref{cor Mimmo} up to the fact that hypothesis (4) and (5) of \Cref{thm:AM2}  are no longer necessary.   Indeed, if (5) is not satisfied then $V$ has no decomposition of the diagonal (automatically, without the need of the other hypothesis) because of \Cref{cime di rape}(1). Similarly, if (4) is not satisfied, then $V$ has no decomposition of the diagonal by  \Cref{cime di rape}(1) combined with \Cref{rottura di cazzo}(2) and \Cref{cime di rape}(3) combined with \Cref{rottura di cazzo}(1).
	\end{proof}
	\begin{corollary}
		Let $V$ be as in Theorem \ref{thm:maincriterioninpractice}. Let 	   $X$ be a smooth proper geometrically connected variety (defined over a field in characteristic zero or two). Suppose that  $X$   degenerates to a variety $Y$ which is birational to $V$  and which has a disingularization $\pi : \tilde{Y}\rightarrow Y $ such that  $\pi$ is universally $\CH_0$-trivial (\Cref{def CH0}). Then $X$   has no  decomposition of the diagonal. 
		\end{corollary}
	 \begin{proof}
	 	Having decomposition of the diagonal is a birational invariant (\Cref{diag vois}) so $\tilde{Y}$ has no decomposition of the diagonal by \Cref{thm:maincriterioninpractice}. We can conclude using \Cref{CTP}.
	 \end{proof}
	
	For conic bundles over surfaces the irrationality criterion can be made easier to apply (see \Cref{basta} for comments).

		\begin{theorem}\label{cor surf}
		Let $k$ be an algebraically closed field of charecteristic two, $S$ be a smooth proper surface   and $f:V\rightarrow S$ be a   flat conic bundle with smooth generic fiber.  Assume the following.
		\begin{enumerate}
			\item The group $\HH^2(S,\Omega^1_{S/k})=0$.
			\item The discriminant divisor is reducible and the singular locus of each irreducible component of the discriminant is contained in the set of points whose fibers are double lines.
			\item The irreducible components of the discriminant meet transversally and the fibers above the intersections   are crosses.
			\item There are at least two irreducible components which are Artin--Mumford in the sense of \Cref{def AM}.
			\item $V$ is smooth around the fibers of double lines.
		\end{enumerate}
		Then $V$ has desingularization $\pi : \tilde{V}\rightarrow V $ such that $\pi$ is universally $\CH_0$-trivial and $\tilde{V}$ has no decomposition of the diagonal.  In particular, any smooth variety (in characteristic zero or two) that degenerates to $V$    has no  decomposition of the diagonal hence it is not stably rational. 
	\end{theorem}
	\begin{proof}
		Consider the points $\{P_i\}$ in $S$ of intersection of two components of the discriminant. Above each point $P_i$ there is a unique point $Q_i$ which is singular in the fiber (as the fiber is a cross). We claim that   the singular points of $V$ are exactly the   $\{Q_i\}$ and that their are ordinary quadratic singularities. Under this claim we can resolve the singularities by simply blowing-up those points. Moreover the exceptional divisors are regular quadrics, hence rational, which implies that the map of desingularization is $\CH_0$-trivial by \cite[Proposition 1.8]{CTPir}. On the other hand the resolution $\tilde{V}$ has no decomposition of the diagonal. Indeed by \Cref{separazione} it is birational to a conic bundle satisying the hypothesis of \Cref{thm:maincriterioninpractice}. We can conclude using \Cref{CTP}.
		
		We have to show the claim. This is essentially   \cite[Theorem 2.14]{Tanaka}. Indeed let us write $t_1,t_2$ for the formal coordinates around a fixed point $P_i$ where the vanishing of a $t_i$ corresponds to an irreducible component of the discriminant. Let $a,b,c$ be the coordinates of the conic bundle. Then the proof in loc. cit. shows that the local equation of  $V$ in the tube above the formal neighborhood around $P_i$ is $\alpha a^2+bc=0$ where $\alpha$ is the discriminant. Hence  in this case we have the equation  $t_1t_2a^2+bc=0$. This implies that the local affine equation around $Q_i$ is $t_1t_2+bc=0$, which gives the claim.
		\end{proof}
		\begin{remark}\label{basta}
				One of the good features of  Theorem \ref{cor surf} with respect to the previous ones, is that  the discriminant does not need to be disconnected. Also the smoothness assumption on the total space is easier to check as it is reduced to the hypothesis (5) (and partially (2)).
			\end{remark}
\section{Concrete examples}
In this section we  construct explicit examples where \Cref{cor surf} applies and deduce from it irrationality results, both in characteristic two and zero (\Cref{teo concreto}). All these examples are in characteristic two and have $\mathbb P^2$ as base. We will use $x,y,z$ for the coordinates on the base $\mathbb P^2$. The points on which the fibers are double lines will always be $[0:1:0]$ and $[0:0:1]$. The total space will be regular above these two points. The discriminant will have two components, each one passing in exactly one of these two points and being singular only there.

Once the examples are found, verifying that they do satisfy these conditions  (and hence the hypothesis of \Cref{cor surf}) is an elementary computation based on the Jacobian criterion and on the description of the loci $\Delta$ and $\Sigma$ (of \Cref{def discr}) using formulas from \Cref{discr car 2}. We do not write these computations.

The way the examples were found was by first guessing the possible discriminants and then finding the coefficients that such a conic bundle should have. With no surprise then, the former have easier equations than the latter. To guess discriminants, it was useful to have the restrictions from \cite[Theorem 4.3]{Auelconic} describing the local behaviour of $\Delta$ around $\Sigma$.

We also write a  last example, taken from \cite{AuelBrauer}, where one of the two components of the discriminant has only crosses but the family over it is not a product.  That example was found using Magma, as the authors explain. The coefficients there are easier than the discriminant.
\begin{example}\label{ex:1}
Let $B=\mathbb P^2$ with coordinates $x,y,z$. Consider the conic bundle defined by
$$
\begin{matrix}\label{unesempioacaso}
&\vline & \mathcal O & \mathcal O(1) &  \mathcal O(3)\\
\hline
\mathcal O  &\vline & 1 &  & \\
\mathcal O (1) &\vline & x & zy&\\
\mathcal O (3) &\vline & 0 & x(y^3+z^3)+y^2z^2&y^6+z^6+x^4yz+xz^5+xy^5\end{matrix}
$$
then the discriminant is
$$\Delta=(x^3z + y^4)(x^3y + z^4).$$
\end{example}
\begin{remark}\label{spieghiamolo}
Let us explain how to find such an example. Write
$$
\begin{matrix}	&\vline & \mathcal O & \mathcal O(1) &  \mathcal O(3)\\
	\hline
	\mathcal O  &\vline & 1 &  & \\
	\mathcal O (1) &\vline & x & \alpha&\\
	\mathcal O (3) &\vline & 0 &\beta & \gamma \end{matrix}
$$
	for a general matrix of this form. The discriminant is $x^2\gamma +\beta ^2$. First notice that $\alpha$ does not appear in the discriminant: its choice will only matter for the smoothness of the total space. Then one is looking for the following two conditions.
	\begin{enumerate}
		\item $\beta=x=0$ consists exactly of the two points $[0:1:0]$ and $[0:0:1]$.
		\item $ (x^3z + y^4)(x^3y + z^4)+ \beta^2$ is divisible by $x^2$: the quotient will then be $\gamma$.
		\end{enumerate}
		Relation (2) simply becomes: $\beta$ is $y^2z^2$ modulo $x$.
			\end{remark}
\begin{example}
Let $B=\mathbb P^2$ with coordinates $x,y,z$ and let  $g$ be an homogenous polynomial of degree 2.
$$
\begin{matrix}
&\vline & \mathcal O & \mathcal O(1) &  \mathcal O(3)\\
\hline
\mathcal O  &\vline & 1 &  & \\
\mathcal O (1) &\vline & x &  g&\\
\mathcal O (3) &\vline & 0 & x(y^3+z^3)+y^2z^2&y^6+z^6+x^4yz+xz^5+xy^5.\end{matrix}
$$
Since the term $g$ does not appear in the equation of the discriminant the factorization 
$$\Delta=(x^3z + y^4)(x^3y + z^4)$$
still holds. 
Hence, every time this conic bundle is smooth around $[0:1:0]$ and $[0:0:1]$ one can apply  \Cref{thm:maincriterioninpractice}. Observe that, the open set of polynomials $g$ satisfying this smoothness condition is nonempty by the previous example.
\end{example}

\begin{example}\label{ex:2}
$B=\mathbb P^2$ with coordinates $x,y,z$. Consider the conic bundle with value in $\mathcal O(1)$  defined by
$$
\begin{matrix}
&\vline & \mathcal O(1) & \mathcal O(1) &  \mathcal O(3)\\
\hline
\mathcal O(1)  &\vline & x+y &  & \\
\mathcal O (1) &\vline & x & x&\\
\mathcal O (3) &\vline &0 & x(y^2+z^2)+y((x+z)z+(z+y)y) & f
\end{matrix}
$$
where $$f=x^2yz^2 + x^2z^3 + xy^4 + xy^3z + xz^4 + y^5 + y^2z^3 + yz^4 +
        z^5.$$ Then the discriminant is
$$\Delta=(x^2z + xy^2 + y^3)(x^2yz + x^2z^2 + y^4 + y^2z^2 + z^4).$$

This example can be found following the strategy from \Cref{spieghiamolo}.
\end{example}
\begin{theorem}\label{teo concreto}
	Let $k$ be a field of characteristic zero or two. A very general conic bundle in $\mathbb P (\mathcal O \oplus \mathcal O(1) \oplus  \mathcal O(3))$ over $\mathbb P_k^2$ with values in $\mathcal O$ has no decomposition of the diagonal, hence it is not stably rational. 

Similarly, a very general conic bundle in $\mathbb P (\mathcal O(1) \oplus \mathcal O(1) \oplus  \mathcal O(3))$ over $\mathbb P_k^2$ with values in $\mathcal O(1)$ has no decomposition of the diagonal, hence  it is not stably rational.
\end{theorem}
By very general we mean that the coefficients of the polynomials defining the conic bundle are algebraically independent over the prime field. 
\begin{proof}
It is enough to degenerate such a very general conic bundle to the above examples \ref{ex:1} and \ref{ex:2} and apply \Cref{cor surf}.
\end{proof}
 
\begin{example}
$B=\mathbb P^2$ with coordinates $x,y,z$. Consider the conic bundle defined by
$$
\begin{matrix}
&\vline & \mathcal O(1) & \mathcal O(1) &  \mathcal O(1)\\
\hline
\mathcal O(1)  &\vline & y^2+j^2z^2 &  & \\
\mathcal O (1) &\vline & xz& yz&\\
\mathcal O (1) &\vline & jx^2+yz & xy & z^2+jy^2
\end{matrix}
$$
with $j$ such that $j^2+j+1=0$. Then the discriminant is
$$\Delta=(x^2z + y^3)(x^2y + z^3)$$
\end{example}
\begin{remark}\label{spieghiamolo ancora}
In order to find such an example, the strategy from \Cref{spieghiamolo} has to be modified because there is no zero in the matrix. The idea is similar, write
	$$
	\begin{matrix}
		&\vline & \mathcal O(1) & \mathcal O(1) &  \mathcal O(1)\\
		\hline
		\mathcal O(1)  &\vline & \alpha &  & \\
		\mathcal O (1) &\vline & \gamma' & \beta&\\
		\mathcal O (1) &\vline & \beta' &\alpha'& \gamma \end{matrix}
	$$
	for a general matrix of this form. First impose that the conics $\alpha',\beta'$and $\gamma'$ meet exactly in $[0:1:0]$ and $[0:0:1]$.
Then complete the matrix so that the discriminant has the desired form $\Delta=(x^2z + y^3)(x^2y + z^3)$. (As already mentioned, such form is guessed based on 
\cite[Theorem 4.3]{Auelconic}.) It is harder to implement this strategy.

An example of this shape might exists over $\mathbb P^3$ (i.e. the quadric surfaces $\alpha',\beta'$and $\gamma'$  meet in a finite number of points and the discriminant has two irreducible components above which only crosses lie.)
This would be of great interest as one would find a variety of dimension four which is irrational because of an $\HH^3$. Arguing through weak Lefschetz one could hope to have examples in any dimension.
\end{remark}

The above example also gives the stable irrationality of some very general conic bundles but this is already in \cite{AuelBrauer}, based on the following example (and a different rationality obstruction).
\begin{example}\label{ex:Auel}
\cite[Section 6]{AuelBrauer}
Consider the conic bundle
$$
\begin{matrix}
&\vline & \mathcal O(1) & \mathcal O(1) &  \mathcal O(1)\\
\hline
\mathcal O(1)  &\vline & xy + xz + z^2 &  & \\
\mathcal O (1) &\vline & x^2+xz+z^2& x^2 + yz + z^2&\\
\mathcal O (1) &\vline & xy &  x^2 + yz + z^2 & y^2 + xz + z^2
\end{matrix}
$$
with discriminant
$$\Delta:=xz(x + z)(y^2x + x^2y + xyz + z^2x + y^3).$$
The fibration is not trivial on $(y^2x + x^2y + xyz + z^2x + y^3)$ by \cite[Lemma 6.8 and Proposition 6.9]{AuelBrauer}. The other three component of the discriminant meet in exactly one point whose fiber is a double line (and the total space is regular there).
\end{example}

 \begin{remark}\label{rem ex double line}
 In characteristic two, it is very easy to construct examples of conic bundles over $\mathbb P^n$ (for any $n$) with reducible discriminant and only double lines over it. It is enough to take a conic bundle of the form
 $$
\begin{matrix}
&\vline & \mathcal O & \mathcal O(n_1) &  \mathcal O(n_2)\\
\hline
\mathcal O  &\vline & 1 &  & \\
\mathcal O (n_1) &\vline & ab& c&\\
\mathcal O (n_2) &\vline & 0&  0& d
\end{matrix} 
$$
and check that the discriminant (with reduced structure) is $ab$ and has the desired properties. For generic choices of $c,d$ the total space of the conic bundle will also be regular (except above the intersection of the two components $a=0$ and $b=0$).
Unfortunately we do not know how to link any of such a construction to \Cref{thm:maincriterioninpractice}. What is missing is a birational transformation allowing to separate the two components $a=0$ and $b=0$, see \Cref{separation double line}.  
\end{remark}
\bibliographystyle{alpha}
\bibliography{bibliografia.bib}

\begin{thebibliography}{ABBGvB21}

\bibitem[ABBGvB21]{Auelconic}
Asher Auel, Alessandro Bigazzi, Christian B\"{o}hning, and Hans-Christian
  Graf~von Bothmer.
\newblock Unramified {B}rauer groups of conic bundle threefolds in
  characteristic two.
\newblock {\em Amer. J. Math.}, 143(5):1601--1631, 2021.

\bibitem[ABBvB21]{AuelBrauer}
Asher Auel, Alessandro Bigazzi, Christian B\"{o}hning, and Hans-Christian~Graf
  von Bothmer.
\newblock Universal triviality of the {C}how group of 0-cycles and the {B}rauer
  group.
\newblock {\em Int. Math. Res. Not. IMRN}, (4):2479--2496, 2021.

\bibitem[AM72]{AM}
M.~Artin and D.~Mumford.
\newblock Some elementary examples of unirational varieties which are not
  rational.
\newblock {\em Proc. London Math. Soc. (3)}, 25:75--95, 1972.

\bibitem[Aue11]{Auelclifford}
Asher Auel.
\newblock Clifford invariants of line bundle-valued quadratic forms.
\newblock {\em MPIM preprint}, 33, 2011.

\bibitem[AV25]{AmbrosiValloni}
Emiliano Ambrosi and Domenico Valloni.
\newblock Reduction modulo p of the {N}oether problem.
\newblock {\em To appear in {A}lgebra and {N}umber {T}heory}, 2025.

\bibitem[Ber74]{Berthelot}
Pierre Berthelot.
\newblock {\em Cohomologie cristalline des sch\'{e}mas de caract\'{e}ristique
  {$p>0$}}.
\newblock Lecture Notes in Mathematics, Vol. 407. Springer-Verlag, Berlin-New
  York, 1974.

\bibitem[CR11]{Rullingbirationalinvariance}
Andre Chatzistamatiou and Kay R\"{u}lling.
\newblock Higher direct images of the structure sheaf in positive
  characteristic.
\newblock {\em Algebra Number Theory}, 5(6):693--775, 2011.

\bibitem[CTO89]{ColliotOj}
Jean-Louis Colliot-Th\'{e}l\`ene and Manuel Ojanguren.
\newblock Vari\'{e}t\'{e}s unirationnelles non rationnelles: au-del\`a de
  l'exemple d'{A}rtin et {M}umford.
\newblock {\em Invent. Math.}, 97(1):141--158, 1989.

\bibitem[CTP16]{CTPir}
Jean-Louis Colliot-Th\'{e}l\`ene and Alena Pirutka.
\newblock Hypersurfaces quartiques de dimension 3: non-rationalit\'{e} stable.
\newblock {\em Ann. Sci. \'{E}c. Norm. Sup\'{e}r. (4)}, 49(2):371--397, 2016.

\bibitem[Ful98]{Ful}
William Fulton.
\newblock {\em Intersection theory}, volume~2 of {\em Ergebnisse der Mathematik
  und ihrer Grenzgebiete. 3. Folge. A Series of Modern Surveys in Mathematics
  [Results in Mathematics and Related Areas. 3rd Series. A Series of Modern
  Surveys in Mathematics]}.
\newblock Springer-Verlag, Berlin, second edition, 1998.

\bibitem[Kat70]{Katz}
Nicholas~M. Katz.
\newblock Nilpotent connections and the monodromy theorem: {A}pplications of a
  result of {T}urrittin.
\newblock {\em Inst. Hautes \'{E}tudes Sci. Publ. Math.}, (39):175--232, 1970.

\bibitem[Kol95]{Kollar}
J\'{a}nos Koll\'{a}r.
\newblock Nonrational hypersurfaces.
\newblock {\em J. Amer. Math. Soc.}, 8(1):241--249, 1995.

\bibitem[Mar82]{Maruyama}
Masaki Maruyama.
\newblock Elementary transformations in the theory of algebraic vector bundles.
\newblock In {\em Algebraic geometry ({L}a {R}\'{a}bida, 1981)}, volume 961 of
  {\em Lecture Notes in Math.}, pages 241--266. Springer, Berlin, 1982.

\bibitem[Mum08]{Mumford}
David Mumford.
\newblock {\em Abelian varieties}, volume~5 of {\em Tata Institute of
  Fundamental Research Studies in Mathematics}.
\newblock Published for the Tata Institute of Fundamental Research, Bombay; by
  Hindustan Book Agency, New Delhi, 2008.
\newblock With appendices by C. P. Ramanujam and Yuri Manin, Corrected reprint
  of the second (1974) edition.

\bibitem[Sch19]{Stefan}
Stefan Schreieder.
\newblock Stably irrational hypersurfaces of small slopes.
\newblock {\em J. Amer. Math. Soc.}, 32(4):1171--1199, 2019.

\bibitem[{Sta}18]{stacks-project}
The {Stacks Project Authors}.
\newblock \textit{Stacks Project}.
\newblock \url{https://stacks.math.columbia.edu}, 2018.

\bibitem[Tan24]{Tanaka}
Hiromu Tanaka.
\newblock Discriminant {D}ivisors for {C}onic {B}undles.
\newblock {\em Q. J. Math.}, 75(4):1301--1353, 2024.

\bibitem[Tot16]{Totaro}
Burt Totaro.
\newblock Hypersurfaces that are not stably rational.
\newblock {\em J. Amer. Math. Soc.}, 29(3):883--891, 2016.

\bibitem[Voi15]{Voisin}
Claire Voisin.
\newblock Unirational threefolds with no universal codimension {$2$} cycle.
\newblock {\em Invent. Math.}, 201(1):207--237, 2015.

\bibitem[Woo09]{Wood}
Melanie Eggers~Matchett Wood.
\newblock {\em Moduli spaces for rings and ideals}.
\newblock ProQuest LLC, Ann Arbor, MI, 2009.
\newblock Thesis (Ph.D.)--Princeton University.

\end{thebibliography}
\end{document}